\documentclass[12pt,reqno]{amsart}

\usepackage{graphicx}
\usepackage{amsmath}
\usepackage{amssymb}
\usepackage{amscd}
\usepackage{hhline}
\usepackage{dcpic}
\usepackage{pictex}
\newtheorem{theorem}{Theorem}

\newtheorem{theoremb}{Theorem}

\newtheorem{theoremd}{Theorem}

\newtheorem{dfn}[theoremb]{Definition}
\newtheorem{rk}[theoremd]{Remark\!\!}
\newtheorem{cor}[theorem]{Corollary}

\newtheorem{prop}[theorem]{Proposition}
\newcommand\bib[1]{\bibitem[#1]{#1}}

\newcommand\1{{\bf 1}}
\renewcommand\a{\alpha}

\newcommand\C{{\mathbb C}}

\newcommand\D{{\mathcal D}}

\newcommand\E{\mathcal{E}}

\renewcommand\l{\lambda}
\newcommand\La{\Lambda}

\newcommand\oo{\omega}
\newcommand\op[1]{\mathop{\rm #1}\nolimits}
\newcommand\ot{\otimes}
\newcommand\p{\partial}

\newcommand\R{{\mathbb R}}

\newcommand\ti{\tilde}

\newcommand\vp{\varphi}

\newcommand\z{\sigma}

\newcommand{\comm}[1]{}

\begin{document}

 \title[Laplace transformation of Lie class $\oo=1$ systems]{Laplace transformation of \\
 Lie class $\oo=1$ overdetermined systems}
 \author{Boris Kruglikov}
 \date{}
 \address{Institute of Mathematics and Statistics, University of Troms\o, Troms\o\ 90-37, Norway. \quad
E-mail: boris.kruglikov@uit.no.}

 \vspace{-14.5pt}
 \begin{abstract}
In this paper we investigate overdetermined systems of scalar PDEs on the plane
with one common characteristic, whose general solution
depends on 1 function of 1 variable. We describe linearization of such systems
and their integration via Laplace transformation, relating this to Lie's
integration theorem and formal theory of PDEs.
 \end{abstract}

 \maketitle

\section*{Introduction}

Consider an overdetermined system $\E$ of partial differential equations,
which we assume formally integrable (i.e. all the compatibility conditions fulfil)
and regular (this is a generic condition micro-locally).
We will restrict to systems with two independent and one dependent variables
(but the theory can be extended to other systems).

The characteristic variety $\op{Char}^\C(\E)$ of this system is an effective
divisor on $\C P^1$, i.e. is a collection of points with positive multiplicities.
Let $\oo=\op{deg}(\op{Char}^\C(\E))$ be the total multiplicity.
This number was called {\it class\/} of the system $\E$ by Sophus Lie \cite{L$_1$}.

In terminology of Ellie Cartan $\oo$ is the {\it Cartan integer\/} $s_1$ (provided the
Cartan character is 1: $s_2=0$). Cartan's test \cite{C$_2$} implies that the general solution
$u\in\op{Sol}(\E)$ depends on $\oo$ functions of 1 variable.

Note that $\oo$ can be described in a different way: Since the characteristic variety is discrete,
the symbol $g_k$ of $\E$ stabilizes: $\lim\limits_{k\to\infty}\dim g_k=\oo$
(when the system becomes involutive, see Appendix \ref{SB} for details).

The case $\oo=0$ corresponds to finite type systems and integration of $\E$ can
be reduced to a system of ODEs via the Frobenius theorem.

Another well-known class is $\oo=2$, especially scalar second order PDEs on the plane.
One of the classical approaches to such systems is the Laplace transformation.

In this paper we show that Laplace transformations exist in the case $\oo=1$ as well
(a comment on the case of general $\oo$ will be made at the end of the paper).

\subsection{Main results}\label{S01}

In his paper of 1895 \cite{L$_1$} S. Lie demonstrated that a compatible
(=formally integrable) class $\oo=1$ overdetermined system is integrable by reduction to ODEs.
Modern proof and applications of this result will be discussed in \cite{K$_2$}.

In this paper we demonstrate that in certain cases this reduction is very precise, and
can be decomposed into a sequence of external transformations in jets that are analogs
of the classical Laplace transformation.
These latter are the differential substitutions with differential inverses.
This provides a method for an effective (algorithmic) integration of PDE systems of class $\oo=1$.

The transformations will be fully described in the linear case, and linearizable systems
will be characterized via a simple criterion.
For non-linear systems we discuss some phenomenology and examples.

Assume that the only characteristic for $\oo=1$ linear system is straightened $X=\p_x$
(this involves integration of a 1st order non-autonomous non-linear scalar ODE).
 \begin{theorem}\label{Thm}
A linear system $\E$ of class $\oo=1$ with one characteristic $X$ is integrable in
closed form and quadratures for generic $\E$,
or is reducible to class $\oo=0$ (ODEs) in singular cases.
 \end{theorem}

Thus a generic linear system of class $\oo=1$ has representation for the general solution
via a differential operator and mild nonlocality (quadrature) applied to an arbitrary function.
This is neither true for for $\oo\ge2$ (e.g. scalar 2nd order PDEs on the plane), nor for $\oo=0$
(e.g. linear 2nd order ODEs, equivalent to Riccati equation).

 \begin{rk}
This distinguishes $\oo=1$ linear systems, but does not extend to nonlinear $\oo=1$ class
systems, which will be discussed in \cite{K$_2$}.
 \end{rk}

Internal geometry of linear/linearizable systems is quite simple:
they correspond to Goursat distributions with growth vector (2,3,4,\dots).
On the other hand the external geometry (which is governed by the pseudogroup of
point triangular transformations) is rich and is characterized by
differential invariants (\cite{C$_1$} contains an example of two systems which
are equivalent internally but not externally, they have type $E_2+E_3$ in the notations of
\S\ref{S2}; E.Cartan proved that this is not possible for $2E_2$). Some of these appear
naturally in our approach, and are the analogs of the classical Laplace invariants.

In this paper we characterize class $\oo=1$ systems from the external viewpoint,
and the most important invariant will be the complexity $\varkappa$ of the system $\E$.
This is an integer-valued non-negative quantity which always decreases under Laplace
transformation. Some other quantities characterizing the size of the solutions space
$\op{Sol}(\E)$ are also relevant.

In the future work \cite{K$_2$}, we shall model reduction types, based on the normal forms of
rank 2 distributions. Cartan-Hilbert equation is a classical example of this general
construction for class $\oo=1$.

\subsection{Background and outline}\label{S02}

We will exploit the geometric theory of PDEs, namely jet-geometry and
Spencer formal theory \cite{S}. The reader is invited to consult \cite{KL$_1$}
or a short exposition in Appendix \ref{SB}.

Also the geometry of distributions will be occasionally used in a minor part
of the text. The structure of the paper is as follows.

In Section \ref{S1} we present a study of class $\oo=1$ linear systems with low-complexity.
We introduce the pseudogroup of frame/coordinate changes and calculate some relative invariants.
Then we define generalized Laplace transformations on phenomenological level as
differential substitutions, which simplify the system (and have differential inverses modulo the equation).

This latter condition means that the Laplace transformations decrease the complexity ---
the notion that is introduced in Section \ref{S2}.
There we also describe the totality (zoo) of all systems of class $\oo=1$
and discuss properties of the complexity.

The generalized Laplace transformations are rigorously defined in Section \ref{S3}.
Then we prove the main results about existence, uniqueness and effectiveness of Laplace transformations
for the linear case.

Section \ref{S4} discusses some features of the non-linear situation. Here new ideas are required:
integrable extensions, non-Moutard form of solutions etc. Some of these will be discussed in
forthcoming paper \cite{K$_2$}, so we restrict to several examples.

In Section \ref{S5} we present a short historical overview and give a brief discussion of
possibilities and difficulties of generalizations of the theory for systems of class $\oo>1$.

The appendices supplement some of the background material.


\medskip

\textsc{Acknowledgment.} It is a pleasure to thank Nail Ibragimov for his translation of
Sophus Lie paper \cite{L$_1$} in \cite{LG}, which was a starting point for this paper.
I am grateful to Valentin Lychagin for many discussions on the initial stage of the project.
The paper is strongly influenced by a collaboration with Ian Anderson, to whom I am thankful.
Hospitality of MLI (Stockholm) in 2007, MFO (Oberwolfach) in 2008,
Banach center (Warsaw) and Utah State University in 2009, as well as IHES in 2010
is acknowledged.

\section{Linear systems on plane with one simple characteristic}\label{S1}

Consider a compatible PDE system $\E$ of class $\oo=1$, meaning there is only
one characteristic\footnote{The characteristic covectors are defined as in
the formal theory of PDEs \cite{S,KL$_1$}, see Appendix \ref{SB}. For 2-dimensional
base they dualize to characteristic vectors (defined up to non-zero multiple).}
of multiplicity 1 (it can be taken $X=\p_x$
for linear and some quasi-linear systems, but not for a fully non-linear $\E$).

The simplest situation is one 1st order PDE, which is classically
known to be solvable via the method of characteristics. In general $\E$ is
generated by $r$ PDEs of possibly different orders and we list only
generators (disregard prolongations) of the system. The number of
these generators can be determined invariantly via the Spencer
$\delta$-cohomology groups: $r=\dim H^{*,1}(\E)$.

Starting from some jet-level we have for the symbol $\dim g_k=1$.
This imposes restriction on the form of the system and we shall
investigate them successively according to complicacy, which
includes the number $r$, orders of the PDEs and orders of the
compatibility conditions.

Spencer cohomology have dimensions $h^0=\dim H^{*,0}=1$, $h^1=\dim
H^{*,1}=r$ and from vanishing of the Euler characteristic\footnote{This holds for
all systems, reducible to ODEs (in our case $\oo=0,1$).
 } we find
$h^2=\dim H^{*,2}=h^1-h^0=r-1$. This is the number of compatibility
conditions our system $\E$ satisfies.

In this section we consider only linear PDEs in one (scalar) unknown $u=u(x,y)$,
in which case the characteristic is a vector field on the base $M=\R^2$.
We will denote it by $X$ and its complement by $Y$.

We take the freedom of considering $X,Y$ to be first order differential operators,
rather than vector fields (and still call it frame), and the symbol of $X$ is the characteristic.

The maximal transformation group $\mathcal{G}$ of this nonholonomic frame is
 $$
X\mapsto \kappa X+a,\quad Y\mapsto \varsigma Y+b X+c.
 $$
We can always arrange $[X,Y]=0$ and in doing so\footnote{With this the pseudogroup
is equivalent to the linear-triangular pseudogroup $x\mapsto\chi(x,y)$, $y\mapsto\gamma(y)$,
$u\mapsto\vartheta(x,y)u$ provided the characteristic is $\p_x$.}
we get functional dimension of the pseudogroup over the algebra $A=C^\infty(M)$ equal 2.

We however relax the requirement
to $\op{ord}[X,Y]=0$ (this is always achieved by a choice of $\kappa,\varsigma$)
and the pseudogroup with this condition has functional dimension $\op{func.dim}_A(\mathcal{G})=3$.

Generators of the pseudogroup $\mathcal{G}$ are the following transformations
 \begin{align*}
&X\mapsto X+a,\ Y\mapsto Y& &\text{(basic gauge)}\\
&X\mapsto e^\l X,\ Y\mapsto Y+bX+c& &\text{(auxiliary gauge)}\\
&X\mapsto \kappa X,\ Y\mapsto \varsigma Y& &\text{(auxiliary gauge)}
 \end{align*}
with $a,b,c\in C^\infty(M)$ arbitrary and $\l,\kappa,\varsigma\in C^\infty(M)$ satisfy
$[Y+bX,\l]=[X,b]$, $[X,\varsigma]=[Y,\kappa]=0$, i.e. $\l,\kappa,\varsigma$ are functions of
one argument (we write $[X,b]$ instead of $X(b)$ as $X$ is a 1st order differential operator
and similarly in other cases).

 \begin{rk}
Using bigger pseudogroup $\mathcal{G}$ instead of point transformations
is harder from the viewpoint of differential invariants, but
is more convenient from factorization viewpoint that is our main goal now.

Note that alternatively we could use the pseudogroup $\mathcal{G}'$: $X\mapsto X+a$, $Y\mapsto Y+bX+c$
(which satisfies the condition $[X,Y]=0\!\mod X+\ \text{order }0$) with $\op{func.dim}_A(\mathcal{G}')=3$,
achieving the same results.
 \end{rk}

Notice that the basic gauge transformations form a normal subgroup $\mathcal{H}$, the auxiliary
form a subgroup $\mathcal{K}$ and we get $\mathcal{H}\cdot\mathcal{K}=\mathcal{G}$. We'll be using the
following easy result:

 \begin{prop}\label{P1}
Let $\mathcal{G}=\mathcal{H}\cdot\mathcal{K}$ be a decomposition of a group into the product
of two subgroups the first of which is normal. Consider an action of $\mathcal{G}$ on a space $M$
(manifold in finite-dimensional case) and suppose there exists a global $\mathcal{H}$-transversal
subspace $L\subset M$ invariant with respect to $\mathcal{K}$. Then the space of invariant
functions $C^\infty(L)^\mathcal{K}$ pulled back to $M$ via $\mathcal{H}$-action coincides
with the space of $\mathcal{G}$-invariants:
 $$
C^\infty(M)^\mathcal{G}\simeq C^\infty(L)^\mathcal{K}.
 $$
 \end{prop}

Now we consider low complexity linear systems of $\oo=1$ type to illustrate the general picture
that we'll sum up later.

\subsection{Two equations of the 2nd order: ${\bf 2E_2}$}\label{S11}

Consider the case of two PDEs ($r=2$) on the plane $M$. As the common characteristic
is\footnote{We write just $X$ instead of more appropriate $\op{symb}(X)$.} $X$,
the equations have symbols $X^2$ and $YX$ (it will be more convenient to write this order
for calculations later on). Thus our system $\E$ is
 $$
\left\{ \begin{array}{l}
 \,X^2\,u+a_1Xu+b_1Yu+c_1u=0,\\
 YXu+a_2Xu+b_2Yu+c_2u=0.
\end{array}\right.
 $$

This system has one compatibility condition since $h^2=1$. It
implies $b_1=0$. By the basic gauge (absorbing functions into $X$)
we transform the system to have $b_2=0$. Then compatibility implies
$c_1=0$.

Namely we have $X\mapsto\tilde X=X+b_2$, and for the transformed $2E_2$
we have $\tilde b_2=\tilde c_1=0$ and
 $$
\tilde c_2=c_2-a_2b_2-[Y,b_2]
 $$
(compare with the classical Laplace invariants, see Appendix \ref{SA}). This $\tilde c_2$
is a relative differential invariant with respect to auxiliary gauge,
and so is a relative invariant of the pseudogroup $\mathcal{G}$ action
(here and in what follows we use Proposition \ref{P1}).

Two sub-cases are possible (we are omitting tildes from now on).

\smallskip

{$\bold \Upsilon_{22}^{{\rm a}}:$} $c_2\ne0$. Then differential
substitution\footnote{Such substitutions, simplifying $\E$, will be called
Laplace transformations.} $v=Xu$ leads us to one first order equation
(essentially parametrized ODE):
 \begin{equation}\label{2E2A1}
(X+a_1)v=0.
 \end{equation}
The second PDE in $\E$ yields the inversion formula
$u=-c_2^{-1}(Y+a_2)v$. Note that in this case the inversion is
differential.

The equation $[Xc_2^{-1}(Y+a_2)+1]v=0$ comes from composition of the
substitution and the inversion, but it is a differential corollary
of (\ref{2E2A1}) due to compatibility of $\E$.

\smallskip

{$\bold \Upsilon_{22}^{{\rm b}}:$} $c_2=0$. Then differential
substitution $v=Xu$ leads us to the Frobenius system:
 \begin{equation}\label{2E2A2}
(X+a_1)v=0,\quad (Y+a_2)v=0,
 \end{equation}
which is compatible (due to compatibility of $\E$) and has
1-dimensional vector space of smooth solutions.

However inversion in this case is integral $u=X^{-1}v$ (note however
that this is a parametrized ODE) providing one unknown function for
the general solution $u$ via essentially one solution of (\ref{2E2A2}).

\subsection{Two equations of the 2nd and 3rd orders: ${\bf E_2+E_3}$}\label{S12}

In this case ($r=2$) a common characteristic condition leads us to two
cases. Since we consider first the lower order PDEs, prolong them,
add new equations and so on, the general position case will be such
that the lower order symbols are in general position. The first of
the cases below is of general position, the second as we shall see
has degeneracy.

\medskip

{$\bold \Upsilon_{23}^1:$} The symbols of the PDEs are $YX$ and
$X^3$. Thus the system $\E$ has the form:
 $$
\left\{ \begin{array}{l}
 YXu+c_1Xu+d_1Yu+e_1u=0,\\
 \,X^3u+a_2X^2u+b_2Y^2u+c_2Xu+d_2Yu+e_2u=0.
\end{array}\right.
 $$
The compatibility condition yields $b_2=d_2=0$ and with the basic
gauge we get $d_1=0$. Then compatibility leads to $e_2=0$.

After this transformation the new coefficient $e_1$ (which is
 $$
\tilde e_1=e_1-c_1d_1-[Y,d_1]
 $$
via the old coefficients and we remove tilde; such precise formulae will be omitted in the future cases)
is a relative invariant with respect to auxiliary gauges
(here the pseudogroup $\mathcal{G}$ is smaller as the symbol of $\E$ fixes directions of
both $X$ and $Y$ thus reducing some of auxiliary gauges), and consequently with respect to $\mathcal{G}$.

The following sub-cases occur:

\smallskip

{$\bold \Upsilon_{23}^{1{\rm a}}:$} $e_1\ne0$. The substitution
$Xu=v$ has differential inverse from the first equation:
$u=-e_1^{-1}(Y+c_1)v$.

The second PDE of $\E$ leads  to a 2nd order PDE for $v$ with symbol
$X^2$. Inserting the inversion into the substitution we get another
PDE with symbol $YX$. The obtained system $2E_2$ is:
 $$
\left\{ \begin{array}{l}
 \bigl(X^2+a_2X+c_2\bigr)v=0,\\
 \bigl(Xe_1^{-1}(Y+c_1)+1\bigr)v=0.
\end{array}\right.
 $$

\smallskip

{$\bold \Upsilon_{23}^{1{\rm b}}:$} $e_1=0$. Then the substitution
$v=Xu$ yields two PDEs, one of the 1st and the other of the 2nd
order, but this system is of Frobenius type (and compatible):
 $$
(Y+c_1)v=0,\quad (X^2+a_2X+c_2)v=0.
 $$
The inversion is however integral: $u=X^{-1}v$.

\medskip

{$\bold \Upsilon_{23}^2:$} The symbols of the PDEs are $X^2$ and
$Y^2X$. Thus the system $\E$ has the form:
 $$
\left\{ \begin{array}{l}
 \,\,X^2\,  u+c_1Xu+d_1Yu+e_1u=0,\\
 Y^2Xu+a_2YXu+b_2Y^2u+c_2Xu+d_2Yu+e_2u=0.
\end{array}\right.
 $$
Compatibility condition yields $d_1=0$ and with the basic gauge we
obtain $b_2=0$. Then the compatibility leads to $e_1=0$.

The new coefficient $d_2$ is a relative differential invariant and $e_2$
is a relative invariant provided $d_2=0$. The following are the sub-cases:

 \comm{
They are expressed via the old coefficients as
 $$
\tilde d_2=d_2-a_2b_2-2[Y,b_2],\quad \tilde e_2=e_2-c_2b_2-a_2[Y,b_2]-[Y,[Y,b_2]]
 $$
(and again we remove tilde)
 }

\smallskip

{$\bold \Upsilon_{23}^{2{\rm a}}:$} $d_2\ne0$. The substitution
$v=Xu$ to the first PDE yields the first order equation
 \begin{equation}\label{E2E3A1}
(X+c_1)v=0.
 \end{equation}
But the inversion is given by the solution of the compatible
Frobenius system formed by the substitution and the second PDE:
 $$
Xu=v,\quad (Y+d_2^{-1}e_2)u=-d_2^{-1}(Y^2+a_2Y+c_2)v.
 $$
Let us remark that integration of this system adds not only one constant, but
also one derivative to the arbitrary function of 1 variable coming from the general
solution of (\ref{E2E3A1}).

\smallskip

{$\bold \Upsilon_{23}^{2{\rm b}}:$} $d_2=0,e_2\ne0$. Here $v=Xu$ has
the differential inversion $u=-e_2^{-1}(Y^2+a_2Y+c_2)v$. First order
PDE (\ref{E2E3A1}) remains the transformed equation with this
substitution.

It seems that insertion of the inversion to the substitution leads
to a 3rd order PDE with the symbol $Y^2X$, but it is a differential
corollary of (\ref{E2E3A1}) due to compatibility of $\E$. Thus we
get only one first order PDE.

\smallskip

{$\bold \Upsilon_{23}^{2{\rm c}}:$} $d_2=0,e_2=0$. Here substitution
$v=Xu$ leads to the Frobenius system of the 1st and 2nd order
equations:
 $$
(X+c_1)v=0,\quad (Y^2+a_2Y+c_2)v=0.
 $$
The inversion operator $u=X^{-1}v$ is integral.

\subsection{Three equations of the 3rd order: ${\bf 3E_3}$}\label{S13}

In this case ($r=3$) $\E$ has one single characteristic $X$ iff its form is:
 $$
\left\{ \begin{array}{l}
 \,\,\,X^3\,u+a_1X^2u+b_1YXu+c_1Y^2u+d_1Xu+e_1Yu+f_1u=0,\\
 YX^2u+a_2X^2u+b_2YXu+c_2Y^2u+d_2Xu+e_2Yu+f_2u=0\\
 Y^2Xu+a_3X^2u+b_3YXu+c_3Y^2u+d_3Xu+e_3Yu+f_3u=0.
\end{array}\right.
 $$

Compatibility implies $c_1=c_2=0$ and the basic gauge yields
$c_3=0$. With this the compatibility gives: $e_1=e_2=0$.

The new coefficients $e_3$, $b_1$, $f_1$ are relative differential invariants
(under auxiliary gauge and so under $\mathcal{G}$), as well as $f_2(\text{upon }f_1=0)$,
$f_3(\text{upon }f_1=f_2=0)$ (which are conditional invariants).

Compatibility ties them so: $f_1=b_1e_3$, $f_2=b_2e_3+[X,e_3]$.
The following are the cases:

\smallskip

{$\bold \Upsilon_{333}^{{\rm a}}:$} $e_3\ne0,b_1\ne0$ $\Leftrightarrow f_1\ne0$.
The fact that $b_1\ne0$ implies that the first two
equations are not compatible\footnote{The term "compatible" means "formally integrable"
\cite{KL$_1$}, and is more restrictive than (formally) "solvable", see Appendix \ref{SB}.}
in itself, and that their compatibility condition
implies the third PDE of $\E$. Thus we can discard the latter.

Substitution $v=Xu$ can be inverted from the first PDE:
$u=-f_1^{-1}(X^2+a_1X+b_1Y+d_1)v$.

A linear combination of the second and the first PDEs from $\E$
yields a 2nd order equation on $v$. Another equation is of the third
order and is obtained by inserting the inversion into the
substitution. Thus we get a system on $v$ of type $E_2+E_3$ ($\Upsilon_{23}^1$):
 $$
\left\{ \begin{array}{l}
\left((YX+a_2X+b_2Y+d_2)-f_1^{-1}f_2(X^2+a_1X+b_1Y+d_1)\right)v=0,\\
\left(Xf_1^{-1}(X^2+a_1X+b_1Y+d_1)+1\right)v=0.
\end{array}\right.
 $$
The symbols of these equations after an auxiliary gauge are $YX,X^3$
and the system is compatible.
 \comm{
Remark that the inversion has differential order 2 totally, but order 1 in $Y$,
so that only 1 derivative is added and the complexity of the general solution increases by 1
with the inverse Laplace transformation.
 }

\smallskip

{$\bold \Upsilon_{333}^{{\rm b}}:$} $e_3\ne0$, $f_1=0$ $\Rightarrow b_1=0$.
Suppose first $f_2\ne0$.

Substitution $v=Xu$ can be inverted from the second PDE:
$u=-L[v]$, where $L=f_2^{-1}(YX+a_2X+b_2Y+d_2)$.

Then the first PDE of $\E$ and insertion of the inversion to the third PDE
yield the compatible system on $v$ of type $E_2+E_3$ ($\Upsilon_{23}^2$):
 $$
 \left\{ \begin{array}{l}
(X^2+a_1X+b_1Y+d_1)v=0,\\
((e_3Y+f_3)L-(Y^2+a_3X+b_3Y+d_3)) v=0.
\end{array}\right.
 $$

\smallskip

{$\bold \Upsilon_{333}^{{\rm c}}:$} $e_3\ne0$, $f_1=0$, $b_1=0$, $f_2=0$.
The substitution $v=Xu$ turns the first two PDEs of $\E$ into $2E_2$:
 $$
(X^2+a_1X+d_1)v=0,\quad (YX+a_2X+b_2Y+d_2)v=0,
 $$
with the inversion found from the compatible system of Frobenius type:
 $$
Xu=v,\quad (Y+e_3^{-1}f_3)u=-e_3^{-1}(Y^2+a_3X+b_3Y+d_3)v.
 $$

\smallskip

{$\bold \Upsilon_{333}^{{\rm d}}:$} $e_3=0$, implying $f_1=f_2=0$.
Further exploration of compatibility implies $b_1f_3=0$. Assume first $f_3\ne0$ $\Rightarrow b_1=0$.

Then the third PDE of $\E$ implies inverse $u=-f_3^{-1}(Y^2+a_3X+b_3Y+d_3)v$ for the substitution $v=Xu$.
The first two PDEs of $\E$ yield the same $2E_2$ as in $\Upsilon_{333}^{{\rm c}}$.

\smallskip

{$\bold \Upsilon_{333}^{{\rm e}}:$} $e_3=0$, $f_1=f_2=f_3=0$.
Then substitution $v=Xu$ leads to the Frobenius system
 $$
 \left\{ \begin{array}{l}
(X^2+a_1X+b_1Y+d_1)v=0,\\
(YX+a_2X+b_2Y+d_2)v=0,\quad (Y^2+a_3X+b_3Y+d_3)v=0.
\end{array}\right.,
 $$
with the integral inverse $u=X^{-1}v$.

\subsection{Two equations of the 3rd order: ${\bf 2E_3}$}\label{S14}

This case ($r=2$) is a bit more complicated since the compatibility
conditions are given by operators of the 2nd order. The
characteristic will again be denoted by $X$. Thus the two symbols
have the form $X\cdot Q_1,X\cdot Q_2$, where $Q_i$ are linearly
independent quadrics on $T^*M$, i.e. elements of $S^2TM$.

Denote by $\Pi^2=\langle Q_1,Q_2\rangle$ the plane in $S^2TM$. Let
$q\in S^2T^*M$ be a non-zero element of $\op{Ann}(\Pi^2)$ and
$\omega_Y\in T^*M$ a non-zero element of $\op{Ann}(X)$. The quadric
$q$ on $TM$ is not a square, since otherwise $Q_1$ and $Q_2$ would
have a common factor, which would be another characteristic.

Depending on the position of $q$ regarding the cone of degenerate
quadrics and the line $\omega_Y^2$ we obtain the following two
different possibilities (classification over $\C$ is by one case
smaller) with $\omega_X$ being a covector, complementary to
$\omega_Y$:
 $$
q=\omega_X^2\mp\omega_Y^2\quad\text{ or }\quad q=\omega_X\omega_Y.
 $$
The corresponding normal forms for the plane $\Pi^2$ generate the
normal forms of the symbol of $\E$. They fix the direction of $Y$,
and the auxiliary group becomes smaller (reduction of $\mathcal{G}$
similar to the case $\Upsilon_{23}^{1}$).

\medskip

{$\bold \Upsilon_{33}^{1}:$} The symbols of the two equations are
$(X^2\pm Y^2)X,YX^2$. The equation has the form
 $$
 \!\left\{ \begin{array}{l}
\ \ \ \ \ YX^2\,u\ \ \ +a_1X^2u+b_1YXu+c_1Y^2u+d_1Xu+e_1Yu+f_1u=0,\\
\!\!\!(X^2\pm Y^2)Xu+a_2X^2u+b_2YXu+c_2Y^2u+d_2Xu+e_2Yu+f_2u=0.
\end{array}\right.
 $$
Compatibility implies that $c_1=0$. The basic gauge yields $c_2=0$
and then compatibility gives $e_1=0$.

Then the new functions $e_2,b_1,f_1$ are relative invariants and
the compatibility ties them as $f_1=b_1e_2$; $f_2$ is a conditional invariant on $e_2=0$.

\smallskip

{$\bold \Upsilon_{33}^{1{\rm a}}:$} $f_1\ne0$ $\Rightarrow$ $e_2\ne0$.
Substituting $Xu=v$ we get the differential inversion $u=-L_1[v]$,
$L_1=f_1^{-1}(YX+a_1X+b_1Y+d_1)$, which being inserted into the
second PDE of $\E$ and the substitution gives two differential
equations, with the notation $L_2=X^2\pm Y^2+a_2X+b_2Y+d_2$:
 \begin{equation}\label{13nov2006a}
 \left\{ \begin{array}{l}
((e_2Y+f_1)L_1+L_2)v=0,\\
(XL_1+1)\,v=0.
\end{array}\right.
 \end{equation}
This system is however not compatible. Indeed, we obtained it with
the help of inversion from the system $Xu=v,(e_2Y+f_2)u=-L_2[v]$.

Writing the compatibility condition for this system, i.e.
multiplying from the left by $(Y+\rho_1)$ and by $(-X+\rho_2)$
respectively and adding, we obtain an equation of the type
$\theta\cdot u=L_3[v]$, where $L_3=XL_2+\dots$ and $\theta\ne0$ because
the compatibility conditions for $\E$ are given via differential operators
of the second (not by first!) order.

Combining the latter equation with the PDE $u=-L_1[v]$ we get a 3rd
order differential equation on $v$ with the symbol $(X^2\pm Y^2)X$,
which is the compatibility condition for system (\ref{13nov2006a}).
Finally taking linear combinations of the derived PDEs we obtain the
system of type $3E_3$ on $v$ with the symbols $X^3,YX^2,Y^2X$,
which is now compatible due to compatibility of $\E$.

\smallskip

{$\bold \Upsilon_{33}^{1{\rm b}}:$} $f_1=0$, $e_2\ne0$. Then
substitution $Xu=v$ together with $(e_2Y+f_2)u=-L_2[v]$ has
Frobenius type and its compatibility condition gives a 3rd order PDE
on $v$. Another PDE on $v$ is given by the first PDE of $\E$ and it
is has 2nd order. This pair $E_2+E_3$ on $v$ has symbols $YX,X^3$
($\Upsilon_{23}^{1}$) and is compatible. The inversion is however
Frobenius.

\smallskip

{$\bold \Upsilon_{33}^{1{\rm c}}:$} $f_1=0$, $e_2=0$, $f_2\ne0$. Then we can get
differential inversion from the second PDE of $\E$ $u=-f_2^{-1}L_2[v]$.
The first PDE of $\E$ becomes the 2nd order PDE $(YX+a_2X+b_2Y+d_2)v=0$, while
substitution of the inversion into $Xu=v$ yields $(Xf_2^{-1}L_2+1)(v)=0$,
which modulo the previous PDE is a 3rd order equation with symbol $X^3$. This is a compatible system
of type $E_2+E_3$ ($\Upsilon_{23}^{2}$).

\smallskip

{$\bold \Upsilon_{33}^{1{\rm d}}:$} $f_1=0$, $e_2=f_2=0$. The
substitution $v=Xu$ reduces the system to a pair of compatible 2nd
order PDEs with symbols $X^2\pm Y^2,YX$. Hence the system is of
finite (Frobenius) type and the substitution has integral inverse.

\medskip

{$\bold \Upsilon_{33}^{2}:$} The symbols of the two equations are
$X^3,Y^2X$ and so the system has the form
 $$
 \left\{ \begin{array}{l}
\ \ X^3\,u+a_1X^2u+b_1YXu+c_1Y^2u+d_1Xu+e_1Yu+f_1u=0,\\
Y^2Xu+a_2X^2u+b_2YXu+c_2Y^2u+d_2Xu+e_2Yu+f_2u=0.
\end{array}\right.
 $$
By compatibility $c_1=0$, basic gauge gives $c_2=0$, then
compatibility implies $e_1=0$.

Then the new functions $e_2,b_1,f_1$ are relative invariants and
the compatibility ties them as $f_1=b_1e_2$; $f_2$ is a conditional invariant on $e_2=0$.

Let $L_1=X^2+a_1X+b_1Y+d_1$, $L_2=Y^2+a_2X+b_2Y+d_2$.

\smallskip

{$\bold \Upsilon_{33}^{2{\rm a}}:$} $f_1\ne0$. The inversion is differential $u=-f_1^{-1}L_1[v]$
(from 1st PDE of $\E$). Insertion of this into the 2nd PDE of $\E$ and into substitution $v=Xu$,
together with addition of compatibility conditions  to the last two
leads to a compatible system $3E_3$ with symbols $X^3,YX^2,Y^2X$.

\smallskip

{$\bold \Upsilon_{33}^{2{\rm b}}:$} $f_1=0$, $e_2\ne0$. The first PDE on $v$ is $L_1[v]=0$.
The inversion can be found from the system
 \begin{equation}\label{pty}
Xu=v,\ (Y+e_2^{-1}f_2)u=-e_2^{-1}L_2v.
 \end{equation}
The compatibility of this system has the form $(Xe_2^{-1}L_2+\dots)v=ku$.

In the case $k=0$ we get the 3rd order PDE on $v$ with symbol $Y^2X$, so the reduced system is
$E_2+E_3$ ($\Upsilon_{23}^2$). The inversion is a Frobenius system.

If $k\ne0$, then the inversion is differential $u=k^{-1}(Xe_2^{-1}L_2+\dots)v$.
Substituting this into the 2nd PDE of (\ref{pty}) yields an equation of 4th order with symbol $Y^3X$.
Thus the reduction is of type $E_2+E_4$.

\smallskip

{$\bold \Upsilon_{33}^{2{\rm c}}:$} $f_1=e_2=0$, $f_2\ne0$.
The inversion is differential $u=-f_2^{-1}L_2[v]$. The first PDE on $v$ is $L_1[v]=0$
and the second is obtained by insertion of the inversion into substitution $Xu=v$.
The symbols of the obtained equations are $X^2,Y^2X$.
They are compatible and have type $E_2+E_3$ ($\Upsilon_{23}^2$).

\smallskip

{$\bold \Upsilon_{33}^{2{\rm d}}:$} $f_1=0$, $e_2=0$, $f_2=0$.
The reduction is a Frobenius system of the second order $L_1[v]=0$, $L_2[v]=0$
and the inversion is integral $u=X^{-1}v$.

\subsection{Examples of calculations}\label{S15}
Let us illustrate integration of linear $\oo=1$ systems for some representative types discussed above.

 \comm{
 \begin{multline*}
u_{xxx} = -\tfrac{2(5x+6y)}{x(x+y)}\,u_{xx} -\tfrac8y\,u_{xy}
-{\tfrac{2(18y^4+34xy^3+21x^2y^2+8x^3y+4x^4)}{x^2y^2(x+y)^2}}\,u_x -\tfrac{24(x+y)}{x^3y}\,u,\\
u_{xxy} = -\tfrac{(x^2+xy+y^2)}{x(x+y)y}\,u_{xx} -{\tfrac{(2x+3y)}{x(x+y)}}\,u_{xy} -{\tfrac{(8x^2y+2x^3+
 5y^3+9xy^2)}{x^2y(x+y)^2}}\,u_x,\\
 u_{xyy} = {\tfrac{y}{8x^2}}\,u_{xx} - {\tfrac{(2x^3+2x^2y+7xy^2+3y^3)}{2x^2y(x+y)}}\,u_{xy} -{\tfrac{3(x+y)}{x^3}}\,u_y{}\hfill\\
{}\hfill +{\tfrac{(-11xy^4+8x^5+16x^4y-22x^2y^3-20x^3y^2-y^5)}{8x^3y^2(x+y)^2}}\,u_x -{\tfrac{3(x+y)}{x^4}}\,u.
 \end{multline*}

It reduces by 3 differential substitution to the one 1st order PDE
$u_x=0$. The composition of inversion formulae gives the solution:
 $$
u=\frac8y\, f'''(y) + \frac{12(2x+y)}{x^2y}\, f''(y)+
 \frac{8(3x+2y)}{x^3y}\,f'(y)+ \frac{2(4x+3y)}{x^4y}\,f(y).
 $$
 }

\medskip

\underline{Example 1.} The following compatible system has type $3E_3$ ($\Upsilon^a_{333}$).
 \begin{multline*}
u_{xxx} = \tfrac{3x+6}{x^2}\,u_{xx}+\tfrac{6y}{x^3}\,u_{xy}+\tfrac{2x-12}{x^3}\,u_x-\tfrac{18}{x^3}\,u\\
u_{xxy} = -\tfrac{4x^2+6x+18}{3xy}\,u_{xx}-\tfrac{6}{x^2}\,u_{xy}+\tfrac{8x^2-6x+36}{3x^2y}\,u_x+\tfrac{18}{x^2y}\,u\\
 u_{xyy} = \tfrac{2x^3+9x^2+45x+54}{9y^2}\,u_{xx}+\tfrac{-5x^2+12x+18}{3xy}\,u_{xy}\\
 +\tfrac{-16x^3+9x^2-36x-108}{9xy^2}\,u_x+\tfrac{3}{y}\,u_y+\tfrac{4x^2-9x-18}{xy^2}\,u
 \end{multline*}

It reduces by differential substitutions to the one 1st order PDE
$u_x=0$ as follows $3E_3\to E_2+E_3\to 2E_2\to E_1$.
The branch is generic, i.e. in each type we have a generic case.

We call the above arrows (generalized) \emph{Laplace transformations}.
They all have differential inverses.
The composition of inversion formulae gives the solution:
 $$
u=9y^3f'''(y)+27xy^2f''(y)+36x^2yf'(y)+16x^3f(y).
 $$

\medskip

\underline{Example 2.} Another compatible $3E_3$ of type $\Upsilon^a_{333}$ is
 \begin{multline*}
u_{xxx} = \tfrac{3x-2y}{x^2}\,u_{xx}+\tfrac{2y}{x^2}\,u_{xy}
   -\tfrac{6}{x^2}\,u_x-\tfrac{6y}{x^3}\,u_y+\tfrac{6}{x^3}\,u\\
u_{xxy} = \tfrac{4x-2y}{x^2}\,u_{xx}+\tfrac{2x^2+2xy}{x^3}\,u_{xy}
   -\tfrac{6}{x^2}\,u_x-\tfrac{6y}{x^3}\,u_y+\tfrac{6}{x^3}\,u\qquad\qquad \\
u_{xyy} = \tfrac{4x-2y}{x^2}\,u_{xx}+\tfrac{2x^2+2xy}{x^3}\,u_{xy}+\tfrac{3}{x}\,u_{yy}
   -\tfrac{6}{x^2}\,u_x-\tfrac{6y}{x^3}\,u_y+\tfrac{6}{x^3}\,u
 \end{multline*}

Now the route of Laplace transformations is $3E_3\to E_2+E_3\to 2E_2$ and it is not generic,
as $E_2+E_3$ has type $\Upsilon^2_{23}$, and $2E_2$ has type $\Upsilon_{22}^{{\rm b}}$, so
this reduces to a Frobenius system of the 1st order and the inverse is integral.
The composition of inversion formulae gives the solution:
 $$
u=x^3 f''(y)-6 x^2 f'(y)+6 x f(y)+C y.
 $$

\medskip

\underline{Example 3.}
Finally consider a system of type $E_2+E_3$, which has inverse Laplace transformation of
the Frobenius ($\oo=0$) type.
 $$
u_{xx}=0,\quad u_{xyy}=\frac{x}{y}u_{xy}-\frac1yu_y.
 $$
Laplace transformation $v=u_x$ reduces this to the equation $v_x=0$, but the inverse is given by
the compatible system
 $$
\left\{\begin{array}{l}
u_x=v,\\ u_y=xv_y-yv_{yy}.
\end{array}\right.
 $$
General solution to equation $E_1$ on $v$ together with integration of the Frobenius system yield
 $$
u=(x+1)\vp(y)-y\vp'(y)+C.
 $$

\begin{rk}
In Examples 2 and 3 we displayed the "happy" cases when the quadratures are expressed in the closed form.
Generally the solution of a Frobenius system cannot even be reduced to the quadratures.
\end{rk}

\section{Classification of types of $\oo=1$ systems}\label{S2}

In this section we split the totality of $\oo=1$ systems into classes, and
discuss transformations between them.
Most results of this section are general and apply equally well to
non-linear $\oo=1$ systems.

\subsection{Zoo for $\oo=1$}\label{S21}

Let us compose a table, where we put compatible systems $\E$ of class
$\oo=1$, and organize its columns by the maximal order $k_\text{\rm max}$
and its rows by the number of equations $r=h^1=\dim H^{*,1}(\E)$.

We introduce the following rule for the choice of generators of $\E$.

Consider the orders of the system: $k_\text{min}=k_1\le\dots\le k_r=k_\text{max}$,
and denote the multiplicities by $m_i=\{\# j: k_j=i\}=\dim H^{i-1,1}(\E)$.
Thus $\E$ is given by $m_{k_1}$ equations $F_{k_1,1},\dots,F_{k_1,m_1}$ of order $k_1$, \dots,
$m_{k_r}$ equations $F_{k_r,1},\dots,F_{k_r,m_r}$ of order $k_r$.

We shall write $\E$ symbolically as $\sum\limits_{i=1}^rE_{k_i}=\sum m_iE_i$,
and call the latter \emph{the type} of $\E$ (implicit: $\oo=1$ common characteristic and compatibility).
The following table shows types for systems of order $k_\text{\rm max}\le 5$.

\

 {
\hskip+25pt
 \begin{tabular}{|c||c|c|c|c|c|}
\hhline{|-||-|-|-|-|-|}
 \begin{minipage}[3]{30pt}
 \ \\ $\underrightarrow{\ k_\text{\rm max}\ }$ \\ $\downarrow h^1\vphantom{\frac{A^A}{A_A}}$ 
\end{minipage} & $1$ & $2$ & $3$ & $4$ & $5$ \\
\hhline{=::=====}
  \begin{minipage}[0]{5pt} ${\ }^{\ }$ \\ 1 \\ \end{minipage}
 & \ $E_1$ \ &  \multicolumn{4}{c|}{} \\
\hhline{=::=====}
  \begin{minipage}[0]{5pt} ${\ }^{\ }$ \\ 2 \\ \end{minipage}
 &  & \,$2E_2$\, &
 \!\!\!${\begin{array}{c} E_2+E_3 \\ 2E_3 \end{array}}$\!\!\!
 & \!\!\!${\begin{array}{c} E_2+E_4 \\ E_3+E_4 \\ 2E_4 \end{array}}$\!\!\!
 & \!\!\!${\begin{array}{c} E_2+E_5 \\ E_3+E_5 \\ E_4+E_5 \\ 2E_5 \end{array}}$\!\!\!
 \\
\hhline{=::~====}
  \begin{minipage}[0]{5pt} ${\ }^{\ }$ \\ 3 \\ \end{minipage}
 & \multicolumn{2}{c|}{ } & $3E_3$ &
 \!\!\!${\begin{array}{c} 2E_3+E_4 \\ E_3+2E_4 \\ 3E_4 \end{array}}$\!\!\!
 & \renewcommand{\arraystretch}{1.1}
 \!\!\!${\begin{array}{c} 2E_3+E_5 \\ E_3+E_4+E_5 \\ 2E_4+E_5
 \\ E_3+2E_5 \\ E_4+2E_5 \\ 3E_5 \end{array}}$\!\!\! \\
\hhline{=::~~===}
  \begin{minipage}[0]{5pt} ${\ }^{\ }$ \\ 4 \\ \end{minipage}
 & \multicolumn{3}{c|}{} & $4E_4$ &
 \renewcommand{\arraystretch}{1.1}
 \!\!\!${\begin{array}{c} 3E_4+E_5 \\ 2E_4+2E_5 \\ E_4+3E_5 \\
 4E_5 \end{array}}$\!\!\! \\
\hhline{=::~~~==}
 \begin{minipage}[0]{5pt} ${\ }^{\ }$ \\ 5 \\ \end{minipage}
 & \multicolumn{4}{c|}{} & $5E_5$ \\
\hhline{|-||-|-|-|-|-|}
 \end{tabular}
 }

\

The general (infinite) table is upper-triangular and the first row consists
of one element only.

The generalized Laplace transformations are arrows between the species in this table, which are directed
toward the top-left corner. Here are sample routes for successive Laplace transformations:
 $$
\dots\to E_2+E_6\to E_2+E_5\to E_2+E_4\to E_2+E_3\to 2E_2\to E_1
 $$
and more complicated:
 \begin{multline*}
\dots\to 2E_4\to 2E_4+E_5\to 3E_4\to 4E_4\to E_3+2E_4\to\\
\to 2E_3+E_4\to 3E_3\to E_2+E_3\to E_2\to E_1.
 \end{multline*}

We will discuss the arrows in more details in Section \ref{S24}.

\subsection{Complexity}\label{S22}

The type notations introduced above suppress a lot of information about $\E$ (branching into sub-cases),
like its symbol class, the order of compatibility conditions etc.
We shall introduce an important integer, as a combination of these data,
characterizing the complexity of a compatible system $\E$ of class $\oo=1$.

Let $g_i$ be the symbols of $\E$ (see Appendix \ref{SB} for details).
Starting from some jet-level the dimensions of these subspaces stabilize: $\dim g_i=1$ for $i\gg1$.

 \begin{dfn}
Complexity of $\E$ is the number $\varkappa=\sum\limits_{i=0}^\infty(\dim g_i-1)$.
 \end{dfn}

The following inequality holds for class $\oo=1$ compatible systems:
$r=h^1\le k_\text{min}\le k_\text{max}$.

 \begin{prop}We have:
$\varkappa\le\sum\limits_{i=1}^r(k_i-i)+(k_1-r)\cdot(k_r-r)$.
Equality is attained for the 'boundary' cases: $r=2$ and $r=k_\text{\rm min}$ $(\Rightarrow$ $r=k_\text{\rm max})$.
 \end{prop}

 \begin{proof}
We start with the case $r=k_1$. Then letting $k_{r+1}=\infty$ we get:
 $$
\dim g_i-1=\left\{\begin{array}{ll}
i, & i<k_1 \\ (k_1-j), & k_j\le i< k_{j+1}
\end{array} \right.
 $$
and so the complexity equals
 $$
\varkappa=\frac{k_1(k_1-1)}2+\sum_{j=1}^{r-1}(k_{j+1}-k_j)(k_1-j)=\sum_{1}^{r}k_j-\frac{r(r+1)}2,
 $$
as required.

For $r<k_1$, $\varkappa$ achieves the maximum value if the system $\sum_{i=1}^s E_{k_i}$
has one characteristic $X$ only on the last step $s=r$, and has more before this;
otherwise $\dim g_i$ decreases more rapidly. In other words reducing the symbolic system by $X$,
the system becomes of finite type only at the highest possible order.

Thus the maximal growth of $\dim g_i$ is achieved for
 $$
\dim g_i-1=\left\{\begin{array}{ll}
i, & i<k_1 \\ (k_1-j), & k_j\le i< k_{j+1} \ (j<r) \\ (k_1-r-j), & i=k_r+j \ (0\le j\le k_1-r)
\end{array} \right.
 $$
and then the direct check yields the desired value of $\varkappa$.

Alternatively dimension of symbols does not change if we assume that the system $\E$ at order $k_\text{max}=k_r$
has more than 1 characteristic $X$, but we add equations of orders $k'_{r+t}=k_r+t$, $1\le t\le k_1-r$,
so that only at the last step a unique characteristic is left.
Thus we can use calculation with $r=k_1$ since $(k'_i-i)=k_r-r$ for $r<i\le k_1$.

Finally if $r=2$, then the reduced by $X$ symbolic system is of complete intersection type and
the equality follows from calculation of the dimension of the solutions space as in \cite{KL$_2$}.
 \end{proof}

Thus complexity $\varkappa$ is bounded via the orders of the system.
On the other hand the amount of types grows super-polynomially with $\varkappa$.

\subsection{How many system types have the same complexity $\varkappa$?}\label{23}

Let us list all systems of low complexity $\varkappa=\sum(\dim g_i-1)\le 6$.

 \begin{itemize}
 \item[$\varkappa=0$.] This is possible only for $E_1$.
 \item[$\varkappa=1$.] Obviously the type $2E_2$.
 \item[$\varkappa=2$.] Only one decomposition (no intermediate zeros are possible)
    $2=1+1$. $E_2+E_3$.
 \item[$\varkappa=3$.] Two decompositions $3=1+1+1=1+2$, we get resp. the types $E_2+E_4$ and $3E_3$.
 \item[$\varkappa=4$.] Two decompositions $4=1+1+1+1=1+2+1$, but the latter splits:
    $E_2+E_5$, $2E_3$ and $2E_3+E_4$ (two cases distinguish by no common
    characteristic and one common characteristic for $2E_3$).
 \item[$\varkappa=5$.] Three decompositions $5=1+1+1+1+1=1+2+1+1=1+2+2$.
    Corresponding types are: $E_2+E_6$, $2E_3+E_5$, $E_3+2E_4$.
 \item[$\varkappa=6$.] Here we have four decompositions
    $6=1+1+1+1+1+1=1+2+1+1+1=1+2+2+1=1+2+3$, the third splits:
    $E_2+E_7$, $2E_3+E_6$, $E_3+E_4$, $E_3+E_4+E_5$, $4E_4$.
 \end{itemize}

Denote by $R(n)$ the number of $\oo=1$ different types $\sum m_iE_i$ having complexity $\varkappa=n$.
We have the following values for this function:

\smallskip

 {
\hskip+25pt
 \begin{tabular}{c|c|c|c|c|c|c|c|c|c|c|c|c}
$n$ & {\it1} & {\it2} & {\it3} & {\it4} & {\it5} & {\it6} & {\it7} & {\it8} & {\it9} & {\it10} & \dots &
 {\it asymptotics}\\
\hline
\vphantom{$\frac{A^A}{Q}$} $R(n)$ &
1 & 1 & 2 & 3 & 3 & 5 & 6 & 9 & 11 & 13 & \dots & $\exp(\pi\sqrt{\l\,n})$
 \end{tabular}
 }

\smallskip

One can prove (via a relation with the number-theoretic partition functions) that the quantity $\l$ defined by
$\sqrt{\l}=\frac{1}{\pi}\cdot\overline\lim\,\log R(n)/\sqrt{n}$ satisfies $\frac13\le\l<\frac23$.

\subsection{Classification of transformations: Phenomenology}\label{S24}

Let us summarize our results of \S\ref{S11}-\ref{S14}. Given a
compatible linear system $\E$ of class $\oo=1$ we have constructed via generalized
Laplace transformation $L:u\mapsto Xu$ a new compatible system $\ti\E$.
We have observed that only three different situations were possible:

 \begin{enumerate}
 \item The transformed system $\ti\E$ is of class $\oo=1$ and the
inverse $L^{-1}$ is a differential operator.
 \item The transformed system $\ti\E$ is of class $\oo=1$, but the
inverse operator $L^{-1}$ is obtained by solving a finite type
system.
 \item The transformed system $\ti\E$ is of class $\oo=0$, but the
inverse operator $L^{-1}$ is integral.
 \end{enumerate}

Note that situation (1) is generic. Finding the inverse in situation
(2) is equivalent to solving a system of ODEs, while in situation
(3) it is given by a parametrized ODE, so that the general solution
depends on one unknown function.

We claim that this is the general pattern, namely we have

 \begin{theorem}\label{123}
For compatible linear systems of class $\oo=1$ the generalized Laplace
transformations can branch into three situations (1), (2) and (3)
above. The first occasion is generic. \qed
 \end{theorem}

This will be demonstrated in the next section by showing that Laplace transformations
decrease the complexity.

We take by definition all systems of class $\oo=0$ to be of lower complexity
than the systems of class $\oo=1$. Thus situation (3), when we
leave the table $\oo=1$ satisfies the claim.

Iteration of Laplace transformations leads either to a system of class $\oo=0$ or to one
equation $E_1$ of class $\oo=1$. Both are reduced to the solution of
ODEs in accordance with Sophus Lie's theorem \cite{L$_1$}.

But for the linear $\oo=1$ case and generic $\E$ we obtained an algorithm for finding
the solutions involving differentiations and quadratures only
(for non-generic $\E$ solutions of ODEs can occurs).

\section{Laplace transformations for linear systems}\label{S3}

Now after lots of examples we give a rigorous definition and proof.

\subsection{Generalized Laplace transformations}\label{S31}

Consider a system $\E=\{F[u]=0\}$ of class $\oo=1$ on the unknown $u=u(x,y)$ generated by a vector
linear differential operator $F=(F_1,\dots,F_r)$. We can choose generators of $\E$
in such a way that the maximal degree of $Y$ in $\op{symb}(F_k)$ strictly increases with $k$
(this is independent of basic/auxiliary gauges).

Then there exists precisely one change $X\mapsto X+a$, $a\in C^\infty(M)$
such that in the decomposition $F_k=\sum\a^k_{j\,i}Y^jX^i$
the number $\max\{j:\exists k\,\a^k_{j\,0}\ne0\}$ is minimal (maximal $j$ corresponds to the minimal $k$).
This fixes the basic gauge, but leaves an auxiliary gauge freedom.

 \begin{dfn}
With $X$ fixed as above the generalized Laplace transformation is the substitution
 $$
u\mapsto v=Xu.
 $$
 \end{dfn}

Let $\mathfrak{I}(\E)=[[F_1,\dots,F_r]]$ be the differential ideal of $\E$.

In order to define the transformed PDE system $\ti\E$ let us consider an
ideal\footnote{$\mathcal{J}$ is surely a left differential ideal, but if we restrict $Y$ by $[X,Y]=0$,
it will be also a right differential ideal.}
$\mathcal{J}=\{\sum_{i>0}\beta_{ij}Y^jX^i\in\mathfrak{I}(\E)\}$,
which factorizes $\mathcal{J}=\mathcal{J}'\cdot X$ for $\mathcal{J}'=\{\sum_{i>0}\beta_{ij}Y^jX^{i-1}\}$.
Let $G=(G_1,\dots,G_l)$ be a vector generator of $\mathcal{J}'$. Then we define $\ti\E=\{G[v]=0\}$.

\smallskip

\emph{Inverse Laplace transformation\/} is defined as a scalar (but apriori not necessarily differential) 
operator $L$ such that $L\cdot X=\1\mod\mathfrak{I}(\E)$.

 \begin{prop}
Inverse operator is unique modulo $\mathfrak{I}(\tilde\E)$.
 \end{prop}

 \begin{proof}
Let $L_1$ and $L_2$ be two inversions for the operator $X$, i.e.
 $$
\1-L_1X,\ \1-L_2X\in \mathfrak{I}(\E)\quad\Longrightarrow\quad (L_1-L_2)X\in\mathfrak{I}(\E).
 $$
This means that $(L_1-L_2)\in\mathcal{J}'$ or equivalently $(L_1-L_2)\in\mathfrak{I}(\tilde\E)$.
 \end{proof}

Next let us consider the existence part.

 \begin{theorem}
For a generic system $\E$ the generalized Laplace transformation $v=Xu$
has a differential inverse operator $u=L[v]$.
 \end{theorem}

 \begin{proof}
Let us prolong $\E$ to the place, where it becomes involutive. This
is the jet-level $k$ such that $\dim g_k=1$. Discard all equations
of order lower than $k$. Then we have $r=k$ compatible PDEs of order
$k$ (type $kE_k$), so that $\varkappa=\frac{(k-1)k}2$.

The symbols of these equations are $X^k,YX^{k-1},Y^2X^{k-2},\dots,Y^{k-1}X$.
Compatibility forces the first $(k-1)$ equations to be free of $Y^{k-1}$ terms.
When the basic gauge is applied, the last PDE has the same property and then
the compatibility implies that the first $(k-1)$ equations contain no $Y^{k-2}$ terms.

This is symbolically shown in the following Young diagram (for $4E_4$),
where we omit all terms in equations $F_j$ of $\E$ except $Y^i$:
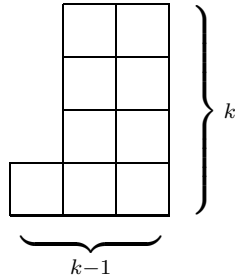
 $$
  \begin{picture}(120,100)
 \put(20,20){\line(0,1){20}}
 \put(40,20){\line(0,1){80}}
 \put(60,20){\line(0,1){80}}
 \put(80,20){\line(0,1){80}}
 \put(20,20){\line(1,0){60}}
 \put(20,40){\line(1,0){60}}
 \put(40,60){\line(1,0){40}}
 \put(40,80){\line(1,0){40}}
 \put(40,100){\line(1,0){40}}
 \put(24,13){$\underbrace{\hphantom{AAAAAA}}_{k-1}$}
 \put(77,57){$\left.\begin{array}{c}\\
 \\ \\ \\ \\ \\ \end{array}\right\}
 {\text{\tiny $k$}}$}
  \end{picture}
 $$
It corresponds to the linear system of $k$ equations and $(k-1)$ unknowns
$Y^{k-2}u,\dots,Yu,u$. Generically coefficients are such that the
rank is maximal. Then we can exclude $Y^{k-2}u,\dots,Yu$ and get an
expression $u=L[v]$, where $L$ is a differential operator of
$\op{ord}\le k-1$.
 \end{proof}

More refined Young diagrams based on the type of symbol of $\E$ can be drawn and then
an inverse $L$ with the minimal possible order (order of the inverse) can be chosen.

For non-generic systems the Laplace transformation can have inverse,
which is either a compatible Frobenius system or a parametrized ODE.
Define order of the inverse in these cases to be $0$ or $-1$ respectively.

\subsection{Proof of the main result}\label{S32}

Theorem \ref{Thm} follows from

 \begin{theorem}\label{SHA}
Under generalized Laplace transformation the complexity strictly decreases.
Generically it decreases only by 1: $\varkappa\mapsto \varkappa-1$.
 \end{theorem}

 \begin{proof}
By re-scaling the unknown function $u\mapsto \sigma\cdot u$ we can achieve $X=\p_x$, and
we adopt this convention in the proof.

Let us interpret complexity $\varkappa$ via Cauchy data, namely a general solution of $\E$
depends on one function of 1 variable, its $q$ derivatives and on $(\varkappa-q)$ constants.
This follows from Cartan-K\"ahler theorem, and generically $q=\varkappa$.
Moreover these are linear superpositions
 $$
u=\a_q f^{(q)}+\dots+\a_1f'+\a_0f+ c_1\phi_1+\dots+c_{\varkappa-q}\phi_{\varkappa-q},
 $$
where $\a_i=\a_i(x,y)$, $\phi_j=\phi_j(x,y)$ are some fixed functions,
$f=f(y)$ is an arbitrary function and $c_k$ are arbitrary constants.

It is easy to check that $\a_q$ depends only on $y$ (this was the reason for
the basic gauge as it is equivalent to $u\mapsto u/\a_q$),
so that the transformation $v=u_x$ reduces $\varkappa$ (only by 1
if $(\a_{q-1})_x\ne0$ and $(\phi_i)_x\ne0$).
 \end{proof}

Generically only transformations (1) from \S\ref{S24} are used and so the
general solution is obtained from the composition of inverse Laplace operators
$u=A[f]$, where $f$ is an arbitrary function of 1 argument and $A$ a linear
differential operator of order $\varkappa$.

Transformations (2) with $L$ solving Frobenius system of order 1 and transformations (3) from \S\ref{S24}
with reduced system of order 1 and class $\oo=0$ (both cases belong to a generic stratum of the space
of singular $\E$) are equally good for Theorem \ref{Thm}, because linear finite type scalar
systems of order 1 are solvable in quadratures.

On the other hand, finite type systems of higher order are usually non-integrable in quadratures
(and a theorem of S.Lie about solvable symmetry groups indicates when such integration is possible).

For example, if the reduced system $\tilde\E$ (or the system for inverse $L$) has type $E_1+E_2$ and class
$\oo=0$:
 $$
\left\{ \begin{array}{r}
 X^2\,u+a_1Xu+b_1u=0,\\
 Yu\,+a_2Xu+b_2u=0,
\end{array}\right.
 $$
then it is generically non-solvable in quadratures. Indeed the first PDE is equivalent to a Riccati equation
via a substitution $u=\exp\int z\,dx$.

Finally note that the algorithm for the Laplace transformation from Section \ref{S31}
together with the first part of the proof of Theorem \ref{SHA} imply that we can perform Laplace
transformations so long the reduction has type $\oo=1$. Theorem \ref{123} follows.

\subsection{The role of relative invariants}\label{S33}

The relative invariants found in the cases of \S\ref{S11}-\ref{S14} play the same fundamental role
as the Laplace invariants, and they govern the branching into sub-cases of the types of $\E$,
determining kind (1), (2) or (3) of the Laplace transformation.

On the other hand we know that these kinds specify the form of the general solution, namely
the amount $(\varkappa-q)$ of constants $c_k$ it contains. These constants are the first integrals
of the system $\E$ and so can be detected by the internal geometry.

Thus the generalized Laplace invariants of $\oo=1$ type systems are the obstructions
to existence of the first integrals.

There can be one such differential invariant
(as in the case $\Upsilon_{23}^1$) or several (as in the case $\Upsilon_{33}^1$),
and there can also occur conditional invariants of various depth (that are defined
only on the stratum given by vanishing of the previous invariants, as in the case $\Upsilon_{23}^2$).

In our approach we first applied Laplace transformations to achieve decrease of the
number $q$ of the involved derivatives, and then had to find a $(\varkappa-q)$-parametric solution
of $\oo=0$ type system.

But equally well we could first restrict to the level leaf of the first integrals (freezing
all the constants) and then solve a generic $\oo=1$ type system.

\bigskip

\section{Towards nonlinear theory}\label{S4}

\subsection{Linearization theorem}\label{S41}

From the internal viewpoint the system $\E\subset J^k(\R^2)$, $k=k_{\max}$, is a submanifold
endowed with the induced distribution $\mathcal{C}_\E=\mathcal{C}\cap T\E$, where
$\mathcal{C}$ is the canonical Cartan distribution on the space of jets.

We assume $\E$ involutive, otherwise we need to prolong it to the jet level $k$
such that $\dim g_k=1$. Then the distribution $\mathcal{C}_\E$ has rank 3 --
it is generated by the vector fields $\D_x$, $\D_y$ (restricted total derivatives)
and the vertical field (kernel of the projection $T\E\to TJ^{k-1}$).

The first of these fields (or its combination with the third field) is a lift of our characteristic $X$
and it coincides with the Cauchy characteristic $\hat X$ for this rank 3 distribution. Denote the (local) quotient
manifold by $\bar\E=\E/\hat X$ and the quotient rank 2 distribution by $\Pi=\mathcal{C}_\E/\hat X$.

Let $\Pi_\infty$ be the bracket-closure of the (generally non-holonomic) distribution $\Pi$.
We assume it is regular. The next claim is obvious.

 \begin{prop}
The amount of constants in the general solution of $\E$ equals $q=\op{codim}(\Pi_\infty)$,
which also coincides with the codimension of the bracket-closure of $\mathcal{C}_\E$ in $\E$. \qed
 \end{prop}

A regular distribution $\Delta_1=\Pi$ generates the following: the weak derived flag $\Delta_i$
via the commutator of sections $\Gamma(\Delta_{i+1})=[\Gamma(\Delta_1),\Gamma(\Delta_i)]$ and
the strong derived flag via $\Gamma(\nabla_{i+1})=[\Gamma(\nabla_i),\Gamma(\nabla_i)]$, $\nabla_1=\Delta_1$.
The sequences of their ranks are called the weak and strong growth vectors respectively.

Goursat distribution is the canonical rank 2 Cartan distribution of the jet-space $J^k(\R)$
(we exclude singularities). If the distribution is not totally-nonholonomic, we can restrict
to the leaves of the foliation of its bracket-closure. Distribution will be called \emph{Goursat-Frobenius}
if it is Goursat in all leaves, i.e. locally the distribution $\mathcal{C}\times0$ on $J^k\R\times\R^m$.

 \begin{theorem}
An involutive $\oo=1$ type system is internally (micro-locally) linearizable if and only if
the corresponding rank 2 (regular) distribution $\Pi$ is Goursat-Frobenius.
 \end{theorem}

 \begin{proof}
It is straightforward to check that both weak and strong flags grow in dimension by 1, so they
are Goursat-Frobenius. The opposite direction is given by Cartan-von Weber theorem,
which states that a rank 2 distribution has locally Goursat normal form
if and only if both weak and strong growth vectors are $(2,3,4,5,\dots)$.
 \end{proof}

As for external linearization, the responsible differential invariants are known only in some
partial cases. The classical case of the 2nd order scalar ODE is due to S.\,Lie and R.\,Liouville
\cite{L$_2$}. For higher order scalar ODEs the contact trivialization result is known \cite{Dou},
and for 3rd and 4th order ODEs the linearization is done (\cite{IMS,SMI} and the references therein),
but to our knowledge no general linearization criterion is known (even for $\oo=0$ type).

For $\oo=1$ systems we expect our generalized Laplace invariants (which exist also in the non-linear situation)
to play an important role in this classification problem.

 \begin{cor}
If a PDE system $\E$ of $\oo=1$ type is effectively linearizable,
then $\E$ possesses a closed form of the general solution.
 \end{cor}

Here 'effectively' stays for algorithmic computability of the linearization transformation.
This means the following. A Goursat distribution can be transformed to the normal form by rectifying some
vector fields. In the presence of symmetry algebra with nice (e.g. solvable) Lie structure this solution
to ODEs can be made effective (for instance in quadratures).

\subsection{Examples of nonlinear Laplace transformations}\label{S42}\qquad

{\bf 1.} Consider the Liouville equation together with one of its Laplace integrals
 \begin{equation}\label{Liou}
u_{xx}=\tfrac12u_x^2,\ u_{xy}=e^u.
 \end{equation}
This quasi-linear system is compatible of class $\oo=1$, and it possesses no intermediate integral.
Laplace transformation here is the same as in the linear case $v=u_x$.

The first PDE of (\ref{Liou}) gives the transformed equation $E_1$:
 $$
v_x=\tfrac12v^2.
 $$
The second equation of (\ref{Liou}) yields the inversion: $u=\log v_y$.

The solution $v=-2/(x+\psi(y))$ to $E_1$ provides the solution to (\ref{Liou}) via inversion:
 $$
u=\log\frac{2\psi'(y)}{(x+\psi(y))^2}.
 $$

Notice that (\ref{Liou}) is not linearizable by an external transformation of Moutard type
(preserving $x,y$), but substitution $w=-2u_x^{-1}-x+2xe^uu_x^{-2}$ makes an internal diffeomorphism of
this system with the linear $2E_2$
 \begin{equation}\label{2E2w}
w_{xx}=0,\ w_{xy}=(w_y-w_x)/x.
 \end{equation}
However according to E.Cartan \cite{C$_1$} a Lie-Backlund type theorem holds in this case,
so that every internal transformation is induced by a contact transformation. We can modify our
external transformation to the following contact equivalence between (\ref{Liou}) and (\ref{2E2w}):
 $$
x\mapsto\frac{u_x}{2e^u},\ y\mapsto y,\ u\mapsto w=-x-\frac1{u_x},\
u_x\mapsto w_x=\frac{2e^u}{u_x^2},\ u_y\mapsto w_y=\frac{u_y}{u_x}.
 $$
In fact, one can show that (\ref{2E2w}) is the normal form for linearizable $2E_2$ without intermediate integrals
(in the latter case we get $u_{xx}=u_{xy}=0$).

 \begin{rk}
The standard B\"acklund transformation that linearizes the Liouville equation maps
(\ref{Liou}) into non-linear system $v_{xx}=\frac12v_x^2$, $v_{xy}=0$.
 \end{rk}

{\bf 2.} Consider another quasi-linear system
 \begin{equation}\label{sqrt2}
u_{xx}=-2\frac{u_x}{x+y},\ u_{xy}=2\frac{\sqrt{u_xu_y}}{x+y}.
 \end{equation}
for which the second equation is also Darboux integrable.

Substitution $u_y=Q^2$ transforms (\ref{sqrt2}) into
 \begin{equation}\label{R41}
Q_{xx}=-\frac{2Q_x}{x+y},\ Q_{xy}=-\frac{Q_x}{x+y}+\frac{Q}{(x+y)^2}.
 \end{equation}
This is a linear $2E_2$ and it has Laplace transform $v=Q_x$ to
 $$
E_1:\ \ v_x=-\frac{2v}{x+y}
 $$
with the inverse $Q=(x+y)^2v_y+(x+y)v$. But the lift (\ref{R41}) to (\ref{sqrt2})
is given by the (compatible) Frobenius system
 $$
u_x=(x+y)^2Q_x^2,\ u_y=Q^2.
 $$

Another nonlinear Laplace transform is given by the following
differential substitution with differential inverse:
 $$
u=w-xw_x-\frac{w_x^2}{x+y}\ \rightleftarrows\ w=u+(x+y)u_x+x(x+y)\sqrt{u_x}.
 $$
This transforms (\ref{sqrt2}) into another quasi-linear system $2E_2$
 $$
w_{xx}=0,\ w_{xy}^2+xw_{xy}-w_y=0.
 $$
This readily yields the general solution $u=\phi(y)-\frac{\psi(y)^2}{x+y}$, $\phi'(y)=\psi'(y)^2$
of (\ref{sqrt2}), and consequently its closed (non-Moutard) form solution:
 $$
x=s,\ \ y=\z''(t),\ \
u=t^2\z''(t)-2t\z'(t)+2\z(t)-\frac{(t\z''(t)-\z'(t))^2}{s+\z''(t)}.
 $$

\medskip

{\bf 3.} Consider fully non-linear system of type $2E_2$
 \begin{equation}\label{ECeq1}
3\,u_{xx}\,u_{yy}^3=-1,\quad u_{xy}=-1/u_{yy},
 \end{equation}
where the first equation belongs to the Goursat's examples of Darboux integrable systems \cite{G}.

This case is more complicated and the nonlinear Laplace transformation is given by $\l=u_{xy}^2$
(notice it is of the 2nd order). The result of the transformation has type $E_1$ and is the equation of gas dynamics
 $$
\l_x=\l\,\l_y.
 $$
The inverse is however not differential, but is given by the Frobenius system
(compatible $3E_2$ of class $\oo=0$):
 $$
u_{xx}=\tfrac13\l^{3/2},\ u_{xy}=\l^{1/2},\ u_{yy}=-\l^{-1/2},
 $$
and thus (\ref{ECeq1}) is solvable by quadratures.
Notice that the choice of $\l$ corresponds to the characteristic of (\ref{ECeq1}).

\subsection{Examples of nonlinear Laplace invariants}\label{S43}
Consider
 $$
 \E: \left\{ \begin{array}{l}
 F(x,y,u,u_x,u_y,u_{xx},u_{xy},u_{yy})=0,\\
 G(x,y,u,u_x,u_y,u_{xx},u_{xy},u_{yy})=0,
\end{array}\right.
 $$
Linearization of this system has the form
 $$
 \ell_\E: \left\{ \begin{array}{l}
 \,X^2\,U+a_1XU+b_1YU+c_1U=0,\\
 YXU+a_2XU+b_2YU+c_2U=0,
\end{array}\right.
 $$
Compatibility implies $b_1=0$. Notice that this does not follow directly from the
arguments for the linear case, because the bracket of linearizations
differs from the linearization of the bracket \cite{K$_1$}. However
we can use the following straightforward statement:
 \begin{prop}
For nonlinear operator $F\in\op{diff}$ and an operator $L$ in total derivatives we have:
$\ell_{L\cdot F}-L\cdot\ell_F=0\mod\langle\D_\z F\rangle$. \qed
 \end{prop}
Now the system $\{F=G=0\}$ for the operators of the above form is compatible iff
there does not exist operators in total derivatives ($C$-differential operators) $L_1,L_2$
such that $\ell_{L_1\cdot F+L_2\cdot G}\in\langle1,X,Y,Y^2\rangle$,
which, by the above proposition, implies $L_1\cdot\ell_F+L_2\cdot\ell_G\in\langle1,X,Y,Y^2\rangle$
on the equation $\E^{(\infty)}$. And this would imply $b_1=0$ and in the case $b_2=0$ also $c_1=0$
for the linearized system $\ell_\E$.

\medskip

The gauge $\bar X=X+b_2$ transforms the system $\ell_\E$ to
 $$
 \left\{ \begin{array}{l}
 \,\bar X^2\,U+\bar a_1\bar XU+\bar c_1U=0,\\
 Y\bar XU+\bar a_2\bar XU+\bar c_2U=0,
\end{array}\right.
 $$
where $\bar a_1=a_1-2b_1$, $\bar c_1=c_1-a_1b_2-\bar X(b_2)=0$
(due to compatibility), $\bar a_2=a_2$ and $\bar c_2=c_2-a_2b_2-Y(b_2)$
(notice similarity of the expressions for $\bar c_1,\bar c_2$ with
the classical Laplace invariants, see Appendix \ref{SA}).

This function $\bar c_2$ is a relative invariant with respect to gauge transformations. It
will be called {\em nonlinear Laplace invariant\/} of the problem.
It vanishes precisely when there exists a (higher) intermediate integral $I$ for the system,
i.e a function $I$ on jets with the vanishing total differential $\hat d(I)=0$ due to the system $\E$.

In our case $\bar c_2=0$ implies existence of an integral $I$.
Indeed, the substitution $V=\bar XU$ reduces $\ell_\E$ to the Frobenius system
 $$
(\bar X+\bar a_1)V=0,\quad (Y+\bar a_2)V=0
 $$
with an intermediate integral $\Phi(x,y,V)=c$, which can be integrated to give
the integral $\phi(x,y,\nabla u)=c$ for the original system.

\medskip

\underline{Example 1.} The Liouville equation coupled with $\D_y$-intermediate integral
is a compatible $2E_2$ of class $\oo=1$:
 $$
u_{xx}-\tfrac12u_x^2=0,\quad u_{xy}-e^u=0.
 $$
Its linearization writes
 $$
X^2U-u_x\,XU=0,\quad YXU-e^uU=0.
 $$
Already in this form $b_1=b_2=c_1=0$ and so $c_2=e^u\ne0$. Consequently this pair
(and the Liouville equation alone) has no intermediate integrals.

\medskip

\underline{Example 2.} Now consider such $2E_2$:
 $$
u_{xx}=u_xe^u,\quad u_{xy}=u_ye^u.
 $$
with linearization
 $$
X^2U-e^uXU-u_xe^uU=0,\quad YXU-e^uYU-u_ye^uU=0.
 $$
In accordance with the above theory $b_1=0$, but $b_2=-e^u\ne0$ (starting from these coefficients
we can calculate by the above formula $\bar c_2=0$). So we make a change of frame
$\bar X=X-e^u$ and get the system
 $$
X\bar XU=0,\quad Y\bar XU=0,
 $$
with an obvious intermediate integral $\bar XU=c$. This linearization can be integrated
to obtain the first integral of the original $2E_2$: $I=u_x-e^u$.

 \medskip

Now let us turn to the case $E_2+E_3$, for which the theory is similar.

\smallskip

\underline{Example 3.} Consider the following linear $\oo=1$ system of type $\Upsilon^{2a}_{23}$
 $$
u_{xx}=0,\quad u_{xyy}-xu_{xy}+u_y=0.
 $$
Laplace transformation $v=u_x$ maps this to $E_1:v_x=0$, and the inverse
can be found from the Frobenius system
 \begin{equation}\label{tre}
u_x=v,\quad  u_y=xv_y-v_{yy}.
 \end{equation}
Integration yields the intermediate integral: $I=u_{xy}-xu_x+u=c$.

 \begin{rk}
Existence of integral $I$ here does not follow from vanishing of a relative differential invariant:
we have $d_2\ne0$ in the notations of $\Upsilon^{2a}_{23}$. But the symbol of this system is degenerate,
which means vanishing of certain relative algebraic invariant.
 \end{rk}

\underline{Example 4.} Now let us consider a nonlinear example:
 $$
u_{xx}=\frac13u_{yy}^3,\quad u_{xyy}=u_{yy}u_{yyy}.
 $$
Its linearization is
 $$
\D_x^2U-u_{yy}^2\D_y^2U=0,\quad \D_x\D_y^2U-u_{yy}\D_y^3U-u_{yyy}\D_y^2U=0.
 $$
Let us introduce the operators $X=\D_x-u_{yy}\D_y$ (characteristic)
and $Y=\D_x+u_{yy}\D_y$. They do not commute $[X,Y]=u_{yyy}(Y-X)$.
Notice also that $X(u_{yy})=0$ and $Y(u_{yy})=2u_{yy}u_{yyy}$ and using this we re-write the above
linearized $E_2+E_3$ system so:
 $$
XYU=0,\quad X^3U=0.
 $$
Notice that $b_2=d_2=e_2=0$ in the notation of $\Upsilon^1_{23}$,
but $d_1=0$ only in this form and $d_1\ne0$ in another representation of the first equation of the system:
 $$
YXU+u_{yyy}(Yu-XU)=0.
 $$
But the coefficient $e_1$ in this normal form is a relative differential invariant and
it controls existence of the integral. For the linearized equation it equals
 $$
\tfrac1{u_{yy}}X^2U=c,
 $$
which can be re-written (modulo equation) as $\D_x\D_yU-u_{yy}\D_y^2U=c$, and this can be integrated
to obtain the intermediate integral of the original problem:
 $$
I=u_{xy}-\tfrac12u_{yy}^2.
 $$
We get the following $\oo=1$ system $2E_2$ that is compatible for any $c$
 $$
u_{xx}=\tfrac13u_{yy}^3,\quad u_{xy}=\tfrac12u_{yy}^2+c.
 $$

 \comm{
Finally let us remark that by dimensional reasons (growth vector)
the case of $2E_2$ with one intermediate integral is always (internally) linearizable,
and the case of $E_2+E_3$ with one intermediate integral is always reduced to
parametric $2E_2$, while if there are 2 i.i. we again get linearization of $E_2+E_3$.
 }

\section{Conclusion: Beyond $\oo=1$ and other generalizations}\label{S5}

The theory of Laplace transformations developed for linear 2nd order PDEs ($\oo=2$)
was further modified to work in the non-linear situation by Darboux and Goursat.

In this paper we mainly concentrated on the general linear $\oo=1$ theory, leaving the
non-linear case for a separate publication. However even the linear theory sheds
a light on the obstructions for $\oo>2$. And indeed this latter case is poor compared to the
theory of $\oo=2$ type.

An important effort to generalize the classical (linear) Laplace theory
to operators of higher order was undertaken in \cite{R}. In particular it was observed
that the Laplace transformation rule applied to one equation inevitably generates several
equations, which was the main obstruction for effective integration theory.

Another attempt was \cite{P}, where transformation theory was constructed for some particular
class of equations, with a degenerate symbol. The reason for this is that starting from 3rd order
the scalar PDEs on the plane have invariants of the symbol, since the latter can be considered
as a planar web. Generic webs have lots of independent differential invariants, and they admit no symmetries.

For recent advances of the factorization part we refer to \cite{T} and
to vast range of papers on differential Galois theory. But the transformation part
has not seen much progress beside the classical 2nd order case (we refer for the
development of the Darboux theory to \cite{AF,AJ,AK}). The reason is that the Laplace
transformation for $\oo>2$ inevitably increases the complexity $\varkappa$ and this gives less chances
of termination for the sequence of Laplace invariants.

Laplace theory, extended for systems of second order equations as the
theory of multidimensional conjugate nets in \cite{D}, was further
developed in \cite{KT,Fe}. This worked for general $\oo$ but again for equations
of very special type, namely semi-Hamiltonian and integrable.

Thus the general theory for $\oo\ge3$ is still lacking, and below we briefly discuss these cases
and $\oo=2$. The case $\oo=1$ turns out, on the contrary, to be a perfect arena for Laplace ideas.

\subsection{Generalized Laplace transformations for $\oo=2$}

The classical method of Laplace concerns $E_2$, we recall it in Appendix \ref{SA}.
The important distinction of this case from $\oo=1$ is that we have a pair of Laplace
transformations $v=Xu$ and $w=Yu$.

Consider the case $2E_3$. A system $\E$ of this type with generic symbol writes
 $$
\left\{ \begin{array}{l}
 YX^2u+a_1X^2u+b_1YXu+c_1Y^2u+d_1Xu+e_1Yu+f_1u=0,\\
 Y^2Xu+a_2X^2u+b_2YXu+c_2Y^2u+d_2Xu+e_2Yu+f_2u=0.
\end{array}\right.
 $$
Compatibility implies $c_1=0$. We begin with Laplace transformation via characteristic $X$.
Then the basic gauge yields $c_2=0$ and compatibility implies $e_1=0$.

With these changes we let $v=Xu$ to be the direct transformation and find the inverse $u=L_1v$
from the 1st PDE of $\E$, where $L_1=-f_1^{-1}(YX+a_1X+b_1Y+d_1)$ and we assume $f_1\ne0$.

Inserting the inversion to the substitution gives the first transformed equation
$(XL_1-1)v=0$, and the second is obtained from the 2nd PDE of $\E$: $(e_2Y+f_2)u=L_2v$,
where $L_2=-(Y^2+a_2X+b_2Y+d_2)$, namely $((e_2Y+f_2)L_1-L_2)v=0$. The transformed
system has symbols $YX^2,Y^2X$ and the type $2E_3$ unless $e_2=0$. In the latter case
the transformed equation is of type $E_2$.

When $f_1=0$, the compatibility gives $e_2=0$. Provided $f_2\ne0$ we get inversion from the
2nd PDE of $\E$, and the 1st PDE gives the transformed $\tilde\E=\{(YX+a_1X+b_1Y+d_1)v=0\}$
of type $E_2$.

In the case $f_1=f_2=0$ the transformation $v=Xu$ brings $\E$ to a system $2E_2$ of type $\oo=1$
and the inverse $u=X^{-1}v$ is integral.

Situation with the second transformation $w=Yx$ is similar.

 \begin{dfn}
Complexity of $\oo=2$ type system is defined by the formula
 $$
\varkappa=\sum_{i=0}^\infty(\dim g_i-2)+1=\sum_{i=2}^\infty(\dim g_i-2).
 $$
 \end{dfn}
The classical case $E_2$ has $\varkappa=0$ and $2E_3$ corresponds to $\varkappa=1$

We can also study the cases of higher complexity $E_3+E_4$ ($\varkappa=2$),
$E_3+E_5$, $3E_4$ (both $\varkappa=3$) etc. forming the zoo of $\oo=2$ type.
These cases are similar and following the considered pattern we get such a conclusion.
 \begin{theorem}
The generalized Laplace transform (any of two) for $\oo=2$ type generically preserves the type
and complexity. In singular cases it can decrease $\varkappa$, but leave $\oo$ or it can reduce the
class to $\oo=1$. \qed
 \end{theorem}

Note that we cannot reach the systems of finite type $\oo=0$ in one step as our Laplace transformations
are always of order 1 (higher order generalizations are plausible).

The integration theory in the general case $\oo=2$ is thus characterized by dropping to class $\oo=1$
type in a sequence of Laplace transformations to both ends. This would reduce solution of $\E$
to ODEs.

If only one sequence of transformations (say $v=Xu$) terminates at $\oo=1$ class, this
gives semi-integrability, i.e. possibility to find many solutions
(a family depending on 1 function of 1 variable) but not all.

\subsection{Generalized Laplace transformations for $\oo\ge3$}

For the class $\oo\ge3$ we define \emph{the complexity} by the formula
 $$
\varkappa=\sum_{i=0}^\infty(\dim g_i-\oo)+\tfrac12\oo(\oo-1)=\sum_{i=\oo}^\infty(\dim g_i-\oo).
 $$
The simplest $\oo=3$ case, corresponding to $\varkappa=0$, is of type $E_3$:
 $$
ZYXu+aX^2u+bYXu+cY^2u+dXu+eYu+fu=0.
 $$
Here $Z$ is a linear combination of $X$ and $Y$. If the web of characteristics
is parallelizable, we can assume $Z=X+Y$. Now there can be 3 Laplace transformations,
with symbol $X$, $Y$ and $Z$ respectively.

Let us study the $X$-transformation. By the basic gauge we achieve $c=0$.
We assume for simplicity that (the new coefficient) $e\ne0$.

Then the Laplace transformation $v=Xu$ leads to $\Psi u=Lv$, where $\Psi=Y+\frac{f}e$, $L=ZY+\dots$
This yields the equation $(XL-\Psi)v=\theta\,u$, where $\theta=[X,\Psi]$ is a function.

Only in the singular case $\theta=0$ has the transformed equation type $E_3$ again:
then $\tilde\E$ is $(XL-\Psi)v=0$ with the same symbol $ZYX$.

In the general case $\theta\ne0$ the above formula gives the differential inversion $u=\theta^{-1}(XL-\Psi)v$
and inserting this into the equation and the substitution we get the following pair of PDEs:
 $$
(ZYX^2+\dots)v=0,\quad (ZY^2X+\dots)v=0.
 $$
Thus the result of the transformation is a system $\tilde\E$ of type $2E_4$ (and class $\oo=3$).
After the next transformation we get equations with higher complexity $\varkappa$, like $3E_5$ or $E_4+E_5$ etc.

 \begin{theorem}
For $\oo=3$ type system any (of three) generalized Laplace transformation generically
increases the complexity by 1. Only in some singular cases it preserves the type
or decreases $\varkappa$ or reduces $\oo$. \qed
 \end{theorem}

For general $\oo$ the complexity increases by $\le \oo-2$ units with any generalized Laplace transformation.
Consequently the integration in closed form is a rare occasion for the class $\oo\ge3$.

\appendix

\section{The method of Laplace}\label{SA}

This method is well-known \cite{D,G,F}. But for completeness, and to show parallel with what is done
for $\oo=1$ case, we give a short review.

The classical case concerns a hyperbolic equation $\E$ of the type $E_2$ (which has class $\oo=2$).
Consider at first the linear case:
 $$
\Delta[u]=u_{xy}+a u_x+b u_y+cu=0,
 $$
where $a,b,c$ are functions on $M=\R^2$. We try to factorize
 $$
\Delta[u]=(\D_y+a)(u_x+b u)-k_0u
 $$
with $k_0=b_y+ab-c$ the 0th level Laplace invariant\footnote{It belongs to a sequence of
(fundamental) relative invariants with respect to linear point transformations
preserving the symbol of $\E$.}. If $k_0=0$, we have an intermediate integral $v=u_x+bu$.

Otherwise the latter formula defines the Laplace transformation with differential inverse
$u=\frac1{k_0}(v_y+av)$. Substitution of the inverse to the direct Laplace transformation
gives the equation
 $$
\Delta_1[v]=(D_x+b)\bigl(\tfrac{v_y+av}{k_0}\bigr)-v=v_{xy}+a_1v_x+b_1v_y+c_1v=0
 $$
Since the equations are effectively equivalent, we search for factorization of this new PDE.
The obstruction for existence of intermediate integral is another relative invariant
$k_1=(b_1)_y+a_1b_1-c_1$.

If the sequence $k_0,k_1,\dots$ stops at zero $k_n=0$, $\E$ is called Darboux semi-integrable.
In this case the equation $\E$ possesses an intermediate integral of order $n$,
which together with the original PDE forms a compatible pair $E_2+E_n$.
Thus semi-integrability can be interpreted as a reduction of $\oo=2$ equation to
a system of class $\oo=1$, which is integrable via ODEs.

The method of Laplace assumes that the sequence of invariants is finite in both sides,
where to the other side we add up invariants\footnote{The invariant $h_n$ can be also interpreted
as an obstruction to find a compatible differential constraint for the hyperbolic equation $\E=\{\Delta[u]=0\}$,
which is a $y$-parametrized ODE of order $n$: $\sum_0^n \a_i u^{(i,0)}=0$.}
$h_0,h_1,\dots$ obtained via $X$-factorization
 $$
\Delta[u]=(\D_x+b)(u_x+a u)-h_0u,
 $$
so that the corresponding Laplace transformation is $v=u_x+a u$ with differential inverse provided
$h_0=a_x+ab-c\ne0$ etc.

 \begin{rk}
The sequence $\dots h_1,h_0,k_0,k_1\dots$ of relative invariants leads to a
collection of absolute invariants, but they do not form a basis
that solves the equivalence problem; some other invariants shall be added \cite{I}.
 \end{rk}

In the non-linear case we cannot proceed with precise transformations, but can compute the sequence of Laplace
invariants anyway via linearizations of the corresponding operators \cite{AK,AJ}.

Darboux integrability, i.e. vanishing of these generalized Laplace invariants to both sides,
again implies closed form of the general solution, as in the linear case. The above interpretation of semi-integrability via reduction of the class from $\oo=2$ to $\oo=1$ is still valid.

\section{Spencer $\delta$-complex and cohomology}\label{SB}

These will be described only in the linear case and for the base $M=\R^2$,
see the general description in \cite{KL$_1$}.

In the case of scalar equations the symbols $g_k$ are the subspaces of the kernels
$\pi_{k,k-1}:J^kM\to J^{k-1}M$, which can be identified with $S^kT^*$,
$T=T_xM$ being the (model) tangent space.

The equation of the lower order $\E_{k_1}$ specifies the subbundle $\E_{k_1}\subset J^{k_1}M$
with the fiber $g_{k_1}$. The next symbols $g_{k+1}$ equals the prolongations
$g_k^{(1)}=(g_k\ot T^*)\cap(S^{k+1}T^*)$ for $k_1\le k<k_2$.

They correspond to the symbols of the prolonged equation $\E_{k_1}$ provided
the compatibility conditions of order $\le k+1$ hold. At the jet-level $k_2$
new equations are added, and we get $g_{k_2}\subset g_{k_2-1}^{(1)}\subset S^{k_2}T^*$ etc.

These symbols are united into the Spencer $\delta$-complex
 \begin{equation}\label{Spencer}
0\to g_{k+1}\to g_k\ot T^*\stackrel{\delta}\longrightarrow g_{k-1}\ot\La^2T^*,
 \end{equation}
where the first arrow is the inclusion and the morphism $\delta$ is the Spencer differential,
i.e. the symbolic exterior derivative (if we interpret $g_k\ot T^*$ as differential 1-forms
on $T$ with polynomial coefficients).

The cohomology group $H^{k,1}(\E)$ at the mid-term of (\ref{Spencer}) counts
(i.e. $\dim H^{k,1}$ equals) the amount of new equations in $\E$ at the jet-level $k$.

The cohomology $H^{k-1,2}(\E)$ at the last term of (\ref{Spencer}) is the most important
--- its elements are the compatibility conditions.

If all the compatibility vanish, the system $\E$ is said to be \emph{formally integrable}.
This is tantamount to the claim that the projections $\pi_{k+1,k}:\E_{k+1}\to\E_k$ are vector
bundles\footnote{Contrary to this, solvability only means that $\E_\infty\to M$ is a (nontrivial) bundle.}
(regularity means constancy of ranks).

Formally integrable systems $\E$ are not necessarily involutive on the level $k=k_{\max}$, but they
are such after some prolongations (for larger $k$). Since involutivity is equivalent to vanishing of the
Spencer cohomology \cite{S} $H^{k,i}=0$, this happens on the jet-level $k$ where $\dim g_k$ stabilizes.

The growth of $\dim g_k$ for large $k$ is given by the Hilbert polynomial
$P_\E(k)=\z\,k^d+\dots$, where $d$ is dimension of the characteristic variety
$\op{Char}^\C(\E)\subset \mathbb{P}^\C T^*$ and $\z$ is its degree \cite{KL$_2$}.

For a scalar system $\E$ the (complex) characteristic variety $\op{Char}^\C(\E)$ is the intersection
of all characteristic varieties for individual equations $F_i=0$ from $\E=\{F_i=0:1\le i\le r\}$ 
(for each of them the characteristics are defined as loci of complexified characteristic polynomial \cite{Pt},
namely of the Fourier transform of the symbol $\op{smbl}(F_i)$).
Alternatively a covector $p\in{}^\C T^*\setminus\{0\}$ is characteristic if $p^k\in g_k^\C$ for all $k\le k_{\max}$.

Notice that for linear systems the variety $\op{Char}^\C(\E)$ depends only on the point $x$ of the base $M=\R^2$.

If $\E$ is determined or overdetermined, then $\op{Char}^\C(\E)$ consists of $\oo=\z$ points
corresponding to (complex) characteristics of $\E$, i.e. $d=0$ and we obtain that $g_k=\oo$ for large $k$.


\end{document}